\documentclass{au}
\usepackage{amssymb,latexsym}
\usepackage{amsmath}
\usepackage[dvipdfm]{graphicx}
 \usepackage{color}
\usepackage{enumerate}
\newenvironment{enumeratei}{\begin{enumerate}[\upshape (i)]}%
                            {\end{enumerate}}

\numberwithin{equation}{section}
\theoremstyle{plain}
 \newtheorem{theorem}{Theorem}[section]
 
 \newtheorem{lemma}[theorem]{Lemma}
 \newtheorem{proposition}[theorem]{Proposition}
 \newtheorem{corollary}[theorem]{Corollary}
\theoremstyle{definition}
 \newtheorem{definition}[theorem]{Definition}
 \newtheorem{remark}[theorem]{Remark}
 \newtheorem*{problem}{Problem}
 \newtheorem{example}[theorem]{Example} 
 \newtheorem*{ackno}{Acknowledgment}
\theoremstyle{remark}
 
\newcommand \init [1] {#1}
\renewcommand \emptyset{\varnothing}
\newcommand \Lat [2] {\textup{JD}(#1,#2)}
\newcommand \cdf [2] {\textup{CDF}(#1,#2)}

\newcommand \upstar [1] {#1^{\ast}}
\newcommand \print [1] {\textup{PrInt}(#1)}
\newcommand \inp {{\mathfrak p}} 
\newcommand \inq {{\mathfrak q}} 
\newcommand \inr {{\mathfrak r}} 
\newcommand \jgen [1] {[#1]_{\mathord\vee}}
\newcommand \blokk [2] {#1/#2}
\newcommand \perm [2] {\pi_{#1#2}}
\newcommand \id {\textup{id}}
\newcommand \then {\mathrel{\Rightarrow}}
\renewcommand \iff {\mathrel{\Leftrightarrow}}
\newcommand \maplp {\xi}
\newcommand \mappl {\eta}
\newcommand \vpi {\vec \pi}
\newcommand \vsigma {\vec \sigma}
\newcommand \elpi {L(\vpi)} 
\newcommand \lancpi [1] {C_{#1}(\vpi)}
\newcommand \tuple [1] {\langle #1 \rangle}
\newcommand \pair [2] {\tuple{ #1,#2}}
\newcommand \tupk [1] {\tuple{#1_1,\ldots,#1_k}}
\newcommand \kset {\set{1,\ldots,k}}
\newcommand \tset {\set{1,\ldots,t}}
\newcommand \nset {\set{1,\ldots,n}}
\newcommand \nnset {\set{0,\ldots,n}}
\newcommand \alg [1] {\mathfrak #1}
\newcommand \isomcl [1] {\mathbf I\kern 0.5 pt #1}
\newcommand \foot [2] { f_{#2}(#1) } 
\newcommand \feet [1] {\vec  f(#1)}
\newcommand \ucsil {u^\ast}
\newcommand \halomd {\textup{LMd}}
\newcommand \halojd {\textup{LJd}}
\newcommand \geom {\textup{Geom}}
\newcommand \amat {\textup{Amat}}
\newcommand \traj [1] {\textup{Traj}(#1)}
\newcommand \Jir [1] {\textup{Ji}\,#1} 
\newcommand \Mir [1] {\textup{Mi}\,#1}
\newcommand \length [1] {\textup{length}\,#1}
\newcommand \dual [1] {{{#1}{}^{\delta}}}
\newcommand \set[1] {\{#1\}}
\DeclareMathOperator{\Ker}{Ker}
\newcommand \restrict [2] {{#1}\kern-1pt \rceil_{\kern-1pt #2}}
\newcommand\ideal[1]{\mathord\downarrow #1}
\newcommand\filter[1]{\mathord\uparrow #1}
\newcommand \bigset[1] {\bigl\{#1\bigr\}}
\newcommand \tbf[1] {\textbf{#1}}       
\renewcommand\rho{\varrho}
\renewcommand\phi{\varphi}

\newcommand \ibvec [1] {\vec b^{(#1)}}
\newcommand \iwvec [1] {\vec w^{(#1)}}

\newcommand\nothing [1] {}

\begin{document}
\title[Coordinatization of join-distributive lattices]
{Coordinatization of join-distributive lattices}
\author[G.\ Cz\'edli]{G\'abor Cz\'edli}
\email{czedli@math.u-szeged.hu}
\urladdr{http://www.math.u-szeged.hu/$\sim$czedli/}
\address{University of Szeged\\Bolyai Institute\\
Szeged, Aradi v\'ertan\'uk tere 1\\HUNGARY 6720}

\thanks{This research was supported by the NFSR of Hungary (OTKA), grant numbers  K77432 and
K83219, and by  T\'AMOP-4.2.1/B-09/1/KONV-2010-0005}


\nothing{
05E99 Algebraic combinatorics (1991-now) None of the above, but in this section 
;
52C99  Discrete Geometry  (1991-now) None of the above, but in this section} 


\date{October 12, 2012; earlier versions:   August 17, 2012,  September 9, 2012}

\begin{abstract} 
\emph{Join-distributive lattices} are finite, meet-semidistributive, and semimodular lattices. 
They are the same as Dilworth's lattices in 1940, and many alternative definitions and equivalent concepts have
been discovered or rediscovered since then. Let $L$ be a join-distributive lattice of length $n$, and let $k$ denote the width of the set of join-irreducible elements of $L$. A result of P.\,H.~Edelman and R.\,E.~Jamison, translated from Combinatorics to Lattice Theory, says that $L$ can be described by $k-1$ permutations acting on the set $\nset$. 
We prove a similar result within Lattice Theory: there exist $k-1$ permutations acting on $\nset$ such that the elements of $L$ are coordinatized by $k$-tuples over $\set{0,\dots,n}$, and the permutations determine which $k$-tuples are allowed.  
Since the concept  of join-distributive lattices is equivalent to that of antimatroids and  convex geometries, 
our result offers a coordinatization for these combinatorial structures. 
\end{abstract}

\maketitle

\section{Introduction}
In 1940, \init{R.\,P.~}Dilworth~\cite{r:dilworth40} introduced an important class of finite lattices. Recently, these lattices are called   join-distributive. The concept of 
antimatroids, which are particular greedoids of \init{B.~}Korte and \init{L.~}Lov\'asz~\cite{kortelovasz81} and \cite{kortelovasz83}, and that of convex geometries were introduced only much later by \init{P.\,H.~}Edelman and \init{R.\,E.~}Jamison \cite{jamison}, \cite{edelman}, and \cite{edeljam}. 
Join-distributive lattices,  antimatroids, and convex geometries are equivalent concepts in a natural way, see Section~\ref{secantimatr}.
 
Hence, though the majority of the paper belongs to Lattice Theory,  the result we prove can also be interesting in Combinatorics. 
Note that there were a lot of discoveries and rediscoveries of join-distributive lattices and the corresponding combinatorial structures; see \init{B.~}Monjardet~\cite{monjardet} and \init{M.~}Stern~\cite{stern} for surveys.

Although there are very deep coordinatization results in Lattice Theory, see \init{J.~}von Neumann~\cite{neumann}, \init{C.~}Herrmann~\cite{herrmann}, and \init{F.~}Weh\-rung~\cite{wehrung} for example, our investigations were motivated  by simple ideas that go back to Descartes. Namely, let $B$ be a  subset  of a  $k$-dimensional Euclidian space $V$,  and  let $\tuple{ v_1,\dots,v_n} \in V^k$ be an orthonormal basis. Then the system $\tuple{ V;v_1,\dots,v_k}$ is represented by $\tuple{ \mathbb R^k; e_1\ldots, e_k}$, where $e_1=\tuple{1, 0,\dots,0} $, \dots, $e_k=\tuple{ 0,\dots,0,1}$, and $B$ corresponds to a subset of $\mathbb R^k$ given by a set of equations, provided $B$ is a ``nice'' subset of $V$. While one can easily describe the relation between two  orthonormal bases of $V$, the analogous task for join-distributive lattices seems to be too hard. This is why we consider the lattice-theoretic counterpart of $\tuple{ V;v_1,\dots,v_k}$ rather than that of $V$. 

Next, instead of $B$, consider a join-distributive lattice $L$ of length $n$. Assume that the width of $\Jir L$, the poset (= partially ordered set) of join-irreducible elements of $L$, equals $k$. Then we can chose $k$ maximal chains, $C_1,\ldots, C_k$, in $L$ such that $\Jir L\subseteq C_1,\ldots, C_k$. These chains will correspond to the vectors $v_i$ above. The direct product $D=C_1\times \cdots\times C_k$, which happens to be the $k$-th direct power of the chain $\set{0<1<\dots<n}$, will play the role of $V$. We know that there is a join-embedding  $\phi\colon L \to D$. If we  describe $\phi(L)$ within $D$ by a simple set of equations, then we obtain a satisfactory description, a coordinatization,  of $L$. These equations will be defined by means of some permutations;  $k-1$ permutations will suffice.  The case $k=2$ was settled, partly rediscovered, in \init{G.~}Cz\'edli and \init{E.\,T.~}Schmidt~\cite{czgschperm}; the case $k > 2$ requires a more complex approach.

\subsection*{On a satellite paper}
After an earlier version of the present paper, available at \texttt{http://arxiv.org/abs/1208.3517}, Kira Adaricheva pointed out that the main result here is closely related to   \init{P.\,H.~}Edelman and  \init{R.\,E.~}Jamison~\cite[Theorem 5.2]{edeljam}, which is formulated for convex geometries. This connection is analyzed in 
\init{K.~}Adaricheva and \init{G.~}Cz\'edli~\cite{adarichevaczg}, which serves as a satellite paper. It appears from \cite{adarichevaczg} that our coordinatization result and the Edelman-Jamison description can mutually be  derived from each other in less than a page. 
However, we feel that the present, almost self-contained, longer  approach still makes sense by the following reasons. 

First,  it exemplifies how Lattice Theory can be applied to other fields of mathematics.
Second, not only the methods  and the motivations of \cite{edeljam} and the present paper are entirely different, the results are not exactly the same; see \cite{adarichevaczg} for comparison. 
Note that our coordinatization is equivalent to a  representation of a join-distributive lattice $L$ as a meet-homomorphic image of the direct power of a chain, while \cite{edeljam} represents $L$ as a join-sublattice of the powerset lattice of the same chain.  While \cite{edeljam} belongs to Combinatorics, 
the coordinatization result is a logical ``step'' in a chain of purely lattice theoretical papers, starting from  \init{G.~}Gr\"atzer and \init{E.~}Knapp~\cite{gratzerknapp} and \init{G.~}Gr\"atzer and \init{J.\,B.~}Nation~\cite{gratzernation}, and including, among others,  \init{G.~}Cz\'edli, \init{L.~}Ozsv\'art, and \init{B.~}Udvari~\cite{czgolub}, \init{G.~}Cz\'edli and \init{E.\,T.~}Schmidt~\cite{czgschsomeres}, \cite{czgschperm}, and also the paper \init{G.~}Cz\'edli and \init{E.\,T.~}Schmidt~\cite{czgschjh}, which gives another application of Lattice Theory.
Third, our method motivates a new characterization of join-distributive lattices, see \cite{edeljam}, and implies some known characterizations, see Remark~\ref
{rwMsRdh}. Fourth, it is not yet clear which approach will be better to attack the problem before Example~\ref{exMple}.

\subsection*{Target}
Let $L$ be a join-distributive lattice of length $n$, and let $C_1,\ldots, C_k$ be maximal chains of $L$ such that $\Jir L\subseteq C_1\cup\dots\cup C_k$. The collection of isomorphism classes of  systems $\langle L;C_1,\ldots,C_k\rangle$ is denoted by $\Lat nk$. The symmetric group of degree $n$, which consists of all 
$\nset \to \nset $ permutations, is denoted by $S_n$. Our goal is to establish a bijection between $\Lat nk$ and $S_n^{k-1}$. If $k$ is small compared to $n$, then this bijection gives a very economic way to describe $\langle L;C_1,\ldots,C_k\rangle$ and, consequently, $L$ with few data. 
Our coordinatization, that is the bijection,
can easily be translated to the language of convex geometries and antimatroids.

\subsection*{Outline} Section~\ref{prelim} contains the lattice theoretic prerequisites, and recalls the known characterizations of join-distributive lattices. Trajectories, which represent the main tool used in the paper,  were introduced for the planar case in \init{G.~}Cz\'edli and \init{E.\,T.~}Schmidt~\cite{czgschjh}. Section~\ref{trajsection} is devoted to trajectories in arbitrary join-distributive lattices. With the help of trajectories, we develop a new approach to Jordan-H\"older permutations in Section~\ref{JHsection}. Our main result, the coordinatization theorem for join-distributive lattices, is formulated in Section~\ref{secmainres}.  This theorem is proved in Section~\ref{sectproofs}. 
Section~\ref{secantimatr} surveys antimatroids and convex geometries briefly. It also translates our coordinatization theorem to the language of Combinatorics.

\section{Preliminaries}\label{prelim}
The objective of this section is to give various descriptions for the lattices  the present paper deals with. 
The \emph{length} of an $(n+1)$-element chain is $n$, while the length of a lattice $L$, denoted by $\length L$, is the supremum of $\{\length C: C$ is a chain of $L\}$. A lattice is \emph{trivial} if it consists of a single element. Let us agree that all lattices in this paper are either finite, or they are explicitely assumed to be of  finite length. 
As usual, $\prec$ stands for the covering relation:  
$x\prec y$ means that the interval $[x,y]$ is 2-element.  If $0\prec a$, then $a$ is an \emph{atom}. A lattice $L$ is \emph{semimodular} if $x\prec y$ implies $x\vee z\prec y\vee z$, for all $x,y,z\in L$. 
An element is \emph{meet-irreducible} if it has exactly one cover. 
The poset of these elements of $L$ is denoted by $\Mir L$. Note that $1\notin \Mir L$ and $0\notin \Jir L$. Since $L$ is of finite length, each element $x\in L$ is of the form $x=\bigwedge Y$ for some $Y\subseteq \Mir L$. Note that $Y=\emptyset$ if{f} $x=1$. The equation $x=\bigwedge Y$ is an \emph{irredundant meet-decomposition} of $x$ if  $Y\subseteq \Mir L$ and   $x\neq \bigwedge Y'$ for every proper subset $Y'$ of $Y$. If each $x\in L$ has only one irredundant meet-decomposition, then we say that $L$ is a lattice \emph{with unique meet-irreducible decompositions}.
A \emph{diamond} of $L$ is a 5-element modular but not distributive sublattice $M_3$ of $L$.  A diamond consists of its top, its bottom, and the rest of its elements form an antichain. If no such sublattice exists, then $L$ is \emph{diamond-free}. 
If $S$ is a sublattice of $L$ such that, for all $x,y\in S$, $x\prec_S y$ implies $x\prec_L y$, then $S$ is a \emph{cover-preserving sublattice} of $L$. If $S$ is a nonempty subset of $L$ such that $x\vee y\in S$ for all $x,y\in S$, then $S$ is a \emph{join-subsemilattice} of $L$. 
For $x\in L$, the join of all covers of $x$ is denoted by $\upstar x$. An important property of $L$ is that $[x,\upstar x]$ is distributive for all $x\in L$.  
If, for all $x,y,z\in L$,  $x\wedge y=x\wedge z$ implies $x\wedge y=x\wedge (y\vee z)$, then $L$ is \emph{meet-semidistributive}. If $\Jir L$ is the union of two chains, then $L$ is \emph{slim}.
The next statement is known and  gives a good understanding of join-distributive lattices within Lattice Theory. For further characterizations, see  Section~\ref{secantimatr} here, see \init{S.\,P.~}Avann~\cite{avann1}, which is recalled in 
 \init{P.\,H.~}Edelman~\cite[Theorem 1.1]{edelmanproc}, and see also 
\init{M.~}Stern~\cite[Theorem 7.2.27]{stern}.

\begin{proposition}\label{Dlwltcs}
For a finite lattice $L$,  the following properties are equivalent.
\begin{enumeratei}
\item\label{Dlwltcsa} $L$ is join-distributive, that is, semimodular and meet-semidistributive
\item\label{Dlwltcsb} $L$ has unique meet-irreducible decompositions. 
\item\label{Dlwltcsc} For each $x\in L$, the interval $[x,\upstar x]$ is distributive.
\item\label{Dlwltcsd} For each $x\in L$, the interval $[x,\upstar x]$ is boolean.
\item\label{Dlwltcse} The length of each  maximal chain of $L$ equals $|\Mir L|$.
\item\label{neWjsldSb} $L$ is semimodular and diamond-free.
\item\label{neWjsldSc} $L$ is semimodular and has no cover-preserving diamond sublattice.
\item\label{neWjsldSd} $L$ is a cover-preserving join-subsemilattice of a finite distributive lattice.
\end{enumeratei}
\end{proposition}

Now, we explain how Proposition~\ref{Dlwltcs} can be extracted from the literature. 
The equivalence of \eqref{Dlwltcsb} and \eqref{Dlwltcsc} above was proved by \init{R.\,P.~}Dilworth~\cite{r:dilworth40}. 
\init{D.~}Armstrong~\cite[Theorem 2.7]{armstrong}  states  \eqref{Dlwltcsa} $\iff$ \eqref{Dlwltcsb} $\iff$ \eqref{Dlwltcsc} by extracting it from  
\init{K.~}Adaricheva, \init{V.\,A.~}Gorbunov, and \init{V.\,I.~}Tumanov~\cite[Theorems 1.7 and 1.9]{r:adarichevaetal}, where the dual statement is given. 
We know \eqref{Dlwltcsc} $\iff$ \eqref{Dlwltcse} $\iff$  \eqref{neWjsldSc}  from \init{M.~}Stern~\cite[Theorem 7.2.27]{stern}, who attributes it to \init{S.\,P.~}Avann~\cite{avann1} and \cite{avann2}.
The implications \eqref{Dlwltcsa} $\then$ \eqref{neWjsldSb} and 
\eqref{neWjsldSb} $\then$ \eqref{neWjsldSc}
 are trivial.  \init{H.~}Abels~\cite[Theorem 3.9]{abels} contains \eqref {neWjsldSc} $\iff$ \eqref {neWjsldSd}. 
Next, as the fourth sentence in \init{P.\,H.~}Edelman~\cite[Section 3]{edelman} points out, 
\eqref{Dlwltcsc} $\iff$ \eqref{Dlwltcsd} is practically trivial; the argument runs as follows. Assume that $[x,\upstar x]$ is distributive. Let $a_1,\ldots,a_t$ be the covers of $x$. They are independent in $[x,\upstar x]$ by distributivity and \init{G.~}Gr\"atzer~\cite[Theorem 360]{GGLT}. Hence they generate a boolean sublattice $B$ of length $t$ and size $2^t$, and $[x,\upstar x]$ is also of length $t$. Since  $|\Jir{([x,\upstar x])}|=\length{([x,\upstar x])}=t$ by \cite[Corollary 112]{GGLT}, we obtain from \cite[Theorem 107]{GGLT} that $|[x,\upstar x]|\leq 2^t$. Thus $[x,\upstar x]=B$ is boolean.

\begin{remark}\label{rwMsRdh}
The proof of our coordinatization result  offers an alternative way to the implication \eqref {neWjsldSc} $\then$ \eqref {Dlwltcsa}, even for lattices of finite length. 
\end{remark}

\begin{remark}\label{reMrkdh} If $L$ is a lattice of finite length, then each of  conditions 
\eqref{Dlwltcsa}, \eqref{Dlwltcsc}, \dots, \eqref{neWjsldSd} implies that $L$ is finite. (Condition \eqref{Dlwltcsb} has not been investigated from this aspect.)  This follows from  either from  Propositions~\ref{Dlwltcs} and \ref{almostmain}, or from Proposition~\ref{Dlwltcs} and  \init{H.~}Abels~\cite[Theorem 3.9]{abels}; see also Corollary~\ref{idsehGH}. 
\end{remark}

\section{Trajectories}\label{trajsection}
The general assumption in Sections~\ref{trajsection} is that $L$ semimodular lattice of finite length and without cover-preserving diamonds. A \emph{prime interval} is a 2-element interval. 
A \emph{covering square} of $L$ is a cover-preserving 4-element boolean sublattice $S=\set{a\wedge b, a,b,a\vee b}$. 
The prime intervals $[a\wedge b, a]$ and $[b, a\vee b]$ are
\emph{opposite sides} of $S$, and so are the prime intervals 
$[a\wedge b, b]$  and $[a, a\vee b]$. The set of prime intervals of $L$ is denoted by $\print L$. If two prime intervals are opposite sides of the same covering square, then they are \emph{consecutive}. As in \init{G.~}Cz\'edli and \init{E.\,T.~}Schmidt~\cite{czgschjh}, the transitive reflexive closure of this consecutiveness relation on $\print L$ is an equivalence relation, and the blocks of this equivalence relation are the \emph{trajectories} of $L$. The collection of all trajectories of $L$ is denoted by $\traj L$. 

\begin{lemma}\label{trajrestlem}
Let $L$ be a semimodular lattice of finite length, having no cover-preserving diamond, and let  $S$ be a cover-preserving join-subsemilattice of $L$. Then the following two statements hold.

$\textup{(A)}$ For each $R\in \traj S$, there is a unique  $T\in\traj L$ such that $R\subseteq T$. 

$\textup{(B)}$  Let $\kappa\colon\traj L\cup\set{\emptyset}\to \traj S\cup\set{\emptyset}$, defined by $\kappa(T)=T\cap\print S$. Then $\kappa$ is a surjective map. If $\set{0_L,1_L}\subseteq S$, then $\kappa$ is a bijection.
\end{lemma}

\begin{proof}[Proof of Lemma~\ref{trajrestlem}$\textup{(A)}$] Denoting the meet in $S$ by $\wedge_S$, let $\set{a\wedge_S b,a,b,a\vee b}$ be a covering square of $S$. Then $a$ and $b$ are incomparable, and both cover $a\wedge_S b$ in $L$ since $S$ is a cover-preserving subset of $L$. This yields $a\wedge_S b=  a\wedge_L b$, and we conclude that the covering squares of $S$ are also covering squares of $L$.  This implies part (A) of the lemma.
\end{proof}

If $L_1$ and $L_2$ are lattices, $\phi\colon L_1\to L_2$ is join-homomorphism, and $\phi(x)\preceq\phi(y)$ holds for all $x,y\in L_1$ with $x\preceq y$, then $\phi$ is a \emph{cover-preserving} join-homomorphism. The kernel of a cover-preserving join-homomorphism is a \emph{cover-preserving join-congruence}. For a join-congruence $\Theta\subseteq L^2$ and a covering square $S=\set{a\wedge b, a,b,a\vee b}$, $S$ is a \emph{$\Theta$-forbidden covering square} if the $\Theta$-blocks $\blokk a\Theta$,  $\blokk b\Theta$, and  $\blokk {(a\wedge b)}\Theta$ are pairwise distinct but $\blokk {(a\vee b)}\Theta$ equals $\blokk a\Theta$ or  $\blokk b\Theta$. The following easy lemma was proved in \init{G.~}Cz\'edli and \init{E.\,T.~}Schmidt~\cite{czgschjh} and \cite[Lemma 6]{czgschsomeres}.

\begin{lemma}\label{thforbIDnn} Let $\Theta$ be a join-congruence of a semimodular lattice of finite length. Then $\Theta$ is cover-preserving if{f} $L$ has no $\Theta$-forbidden covering square.
\end{lemma}

The initial idea of the present paper is formulated in the next lemma.

\begin{lemma}\label{trajlemma}
Let  $L$ be a semimodular lattice of finite length such that $L$ contains no cover-preserving diamond. Then, for each maximal chain $C$ of $L$ and for each trajectory $T$ of $L$, $T$ contains exactly one prime interval of $C$. 
\end{lemma}

\begin{proof} Take a prime interval $\inp\in T$, and pick a maximal chain $D$ of $L$ such that $\inp\in\print D$. Let $S=\jgen{C\cup D}$, the join-subsemilattice generated by $C\cup D$. We know from \init{G.~}Cz\'edli and \init{E.\,T.~}Schmidt~\cite[Lemma 2.4]{czgschjh} that $S$ is a cover-preserving 0,1-join-subsemilattice of $L$, and it is a slim semimodular lattice. (But $S$ is not a sublattice of $L$ in general.) It follows from \cite[Lemmas 2.4 and 2.8]{czgschjh} that there are a unique $\inq\in\print C$ and a unique trajectory $R$ of $S$  such that $\inp$ and $\inq$ belong $R$.  By Lemma~\ref{trajrestlem}(A), $L$ has a trajectory $R'$ such that $R\subseteq R'$. Since $p\in R'\cap T$, we obtain that $T=R'$ contains $\inq$. This proves the existence part. 

Note that, instead of \cite{czgschjh}, one could use  \init{H.~}Abels~\cite[Corollary 3.3]{abels}
to prove the  existence part. Actually, some ideas of \cite{abels} were rediscovered in  \cite{czgschjh}. However, since the concept of trajectories comes from \cite{czgschjh}, it is more convenient to reference \cite{czgschjh}.

We prove the uniqueness by contradiction. Suppose that $a\prec b\leq c\prec d$ such that $[a,b]$ and $[c,d]$ belong to the same trajectory $T$ of $L$. Then there exists a sequence 
$[a,b]=[x_0,y_0]$, $[x_1,y_1]$, \dots, $[x_t,y_t]=[c,d]$
of prime intervals such that $H_u=\set{x_{u-1}, y_{u-1}, x_u,y_u}$ is a covering square for $u\in \tset $. 
Pick a maximal chain $0=b_0\prec b_1\prec\dots\prec b_s=b$ in the interval $[0,b]$, and consider the join-homomorphisms $\psi_m\colon L\to L$, defined by $\psi_m(z)=b_m\vee z$, for $m\in\set{0,\ldots,s}$. Let $\Theta_m$ stand for $\Ker {\psi_m}$, for $m\in\set{0,\ldots,s}$. By semimodularity, the $\psi_m$ are cover-preserving join-homomorphisms. 
Since  $\psi_s(x_0)=b=\psi_s(y_0)$ but    $\psi_s(x_t) =c  \neq d= \psi_s(y_t)$, there is a smallest $i\in\set{1,\ldots,t}$ such that $\psi_s(x_i)\neq \psi_s(y_i)$. 
That is, $\pair{x_{i-1}}{y_{i-1}}\in\Theta_s$ but $\pair{x_{i}}{y_{i}}\notin\Theta_s$. Clearly, $x_{i}$ is the bottom of  the covering square $H_i$, while $y_{i-1}$ is its top. 

The restriction of a relation $\rho$ to a subset $X$ will be denoted by $\restrict\rho X$. The equality relation on $X$ is denoted by $\omega_X$ or $\omega$.  
Since $\restrict{\Theta_0}{H_i} =\omega_{H_i}$ but $\pair{x_{i-1}}{y_{i-1}} \in  \restrict{\Theta_s}{H_i}$, there is a smallest $j$ such that $\pair{x_{i-1}}{y_{i-1}}\in  \restrict{\Theta_j}{H_i}$. 
However, $\pair{x_{i}}{y_{i}}\notin  \restrict{\Theta_j}{H_i}$
since otherwise $b_j\vee x_{i}= \psi_j(x_{i}) = \psi_j(y_{i}) = b_j\vee y_{i}$ together with $b_j\leq b_s$ would imply  $\psi_s(x_{i}) = \psi_s(y_{i})$, which would contradict 
 $\pair{x_{i}}{y_{i}}\notin\Theta_s$. 

Next, to simplify our notation, let
let  
\begin{equation*}
\begin{aligned}
\alpha&=\psi_{j-1}(x_{i}),\cr 
\alpha'&=\psi_{j}(x_{i}),
\end{aligned}
\quad 
\begin{aligned}
 \delta  &= \psi_{j-1}(y_{i-1}),\cr 
 \delta' &= \psi_{j}(y_{i-1}),
\end{aligned}
\quad 
\begin{aligned}
\set{\beta,\gamma} &=\set{\psi_{j-1}(y_{i}), 
 \psi_{j-1}(x_{i-1})} , \cr
 \set{\beta',\gamma'} &=\set{\psi_{j}(y_{i}), 
 \psi_{j}(x_{i-1})}
\end{aligned}
\end{equation*} 
such that $\beta'=\beta\vee b_j$ and $\gamma'=\gamma\vee b_j$. 
By the minimality of $j$, $|\set{\alpha,\beta,\gamma,\delta}|=4$. Hence, $\set{\alpha,\beta,\gamma,\delta}=\psi_{j-1}(H_i)$ is a covering square with bottom $\alpha$ and top $\delta$ in the filter $\filter {b_{j-1}}=[b_{j-1},1]$ since $\psi_{j-1}$ is a cover-preserving join-homomorphism. Consider the cover-preserving join-homomorphism $\phi\colon \filter {b_{j-1}} \to \filter {b_{j-1}}$, defined by $\phi(z)=b_j\vee z$. Denote the height function on $\filter {b_{j-1}}$ by $h$. 
Since $b_j$ is an atom of the filter $\filter b_{j-1}$, semimodularity implies that $h(z)\leq h(\phi(z)) \leq h(z)+1$ holds for all $z\in \filter {b_{j-1}}$. By definitions, $\phi(\alpha)=\alpha'$, $\phi(\beta)=\beta'$, $\phi(\gamma)=\gamma'$, and  $\phi(\delta)=\delta'$, and we also  have $\gamma'=\delta'$ and $\alpha'\neq\beta'$. Let $\Phi=\Ker\phi$. It is a cover-preserving join-congruence, and Lemma~\ref{thforbIDnn} yields that $\set{\beta,\gamma}\subseteq \blokk{\delta}\Phi$ but $\alpha\notin \blokk{\delta}\Phi$.
Thus $\alpha'\neq\beta'=\gamma'=\delta'$. Actually, $\alpha'\prec\beta'=\gamma'=\delta'$  since $\phi$ is cover-preserving. 

Since $\beta\neq\gamma$, we have $\pair{\beta'}{\gamma'}\neq \pair{\beta}{\gamma}$. Hence we can assume $\beta'\neq\beta$, and we obtain $\beta\prec \beta'$. Using 
$h(\gamma)+1=h(\beta)+1=h(\beta')=h(\gamma')$, we obtain $\gamma\prec \gamma'$. Now $h(\delta')=h(\beta')= h(\beta)+1=h(\delta)$, together with $\delta\leq\delta'$, yields $\delta'=\delta$. We have $\alpha'\preceq \beta'=\delta'=\delta$ since $\phi$ is cover-preserving. Hence $\alpha'\neq\alpha$, and we obtain $\alpha\prec \alpha'$.  

The previous relations imply $b_j\leq\delta$, $b_j\not\leq \alpha$,  $b_j\not\leq \beta$, and  $b_j\not\leq \gamma$. Let $p=\alpha\vee b_j$. Since $b_j$ is an atom in the filter $\filter {b_{j-1}}$, we have $\alpha\prec p$. We obtain $p\not\leq \beta$ from $b_j\not\leq \beta$, and we obtain $p\not\leq \gamma$ similarly. Hence $\beta$, $p$, and $\gamma$ are three different covers of $\alpha$ in the interval $[\alpha,\delta]$ of length 2. Thus $\set{\alpha, \beta, p,\gamma,\delta}$ is a cover-preserving diamond of $L$, which is a contradiction.
\end{proof}

Now we are in the position to complete the proof of Lemma~\ref{trajrestlem}.

\begin{proof}[Proof of Lemma~\ref{trajrestlem}$\textup{(B)}$] 
Let $C$ be a map from $\print S$ to the set of maximal chains of $S$ such that, for every $\inp\in\print S$,  $\inp\in\print {C(\inp)}$. 

To show that $\kappa$ is a map from $\traj L$ to $\traj S$, let  $T\in\traj L$ such that $T\cap\print S\neq\emptyset$. Pick a prime interval $\inp\in T\cap\print S$, and let $R\in\traj S$ be the unique trajectory containing $\inp$. At present, we know that 
\begin{equation}\label{oasiwWE}
\text{$T\in\traj L$, $R\in\traj S$, and $p\in T\cap R\cap\print S$,
}
\end{equation}
and this will be the only assumption on $T$ and $R$ we use in the rest of the present paragraph. 
By Lemma~\ref{trajrestlem}$\textup{(A)}$, there is an $R'\in \traj L$ such that $R\subseteq R'$. Since $\inp\in T\cap R'$, we have $R\subseteq R'=T$, which implies $R\subseteq\kappa(T)$. To show the converse inclusion, let $\inq\in\kappa(T)=T\cap\print S$. Applying the existence part of  Lemma~\ref{trajlemma} to $S$, we obtain an $\inr\in R\cap \print{C(\inq)}$. 
Since $R\subseteq T$, both $\inr$ and $\inq$ belong to $T$. Thus the uniqueness part of  Lemma~\ref{trajlemma} gives $\inq=\inr\in R$. Hence, $R=\kappa(T)$, and $\kappa$ is a map from $\traj L$ to $\traj S$.

Since \eqref{oasiwWE} implies $R=\kappa(T)$,   $\,\kappa$ is surjective. 

Finally, let $\set{0_L,1_L}\subseteq S$. Then $S$ contains a maximal chain $X$ of $L$. Assume that $T_1,T_2\in\traj L$ such that $\kappa(T_1)=\kappa(T_2)$. By Lemma~\ref{trajlemma}, there is a unique $\inq\in T_1\cap \print X$. 
We have $q\in T_1\cap\print S= \kappa(T_1)=\kappa(T_2)\subseteq T_2$. Hence $T_1\cap T_2\neq\emptyset$, and we conclude $T_1=T_2$. Consequently, $\kappa$ is injective, and we obtain that it is  bijective.
\end{proof}

\section{Jordan-H\"older permutations}\label{JHsection}
Any two maximal chains of a semimodular lattice of length $n$ determine a so-called  Jordan-H\"older permutation on the set $\set{1,\dots,n}$. This was first stated by \init{R.\,P.~}Stanley~\cite{stanley}; see also \init{H.~}Abels~\cite{abels} for further developments. Independently, the same permutations emerged in \init{G.~}Gr\"atzer and \init{J.\,B.~}Nation~\cite{gratzernation}. The Jordan-H\"older permutations were rediscovered in \init{G.~}Cz\'edli and \init{E.\,T.~}Schmidt~\cite{czgschjh}, and were successfully applied to add a uniqueness part to the classical Jordan-H\"older Theorem for groups. As an excuse for this rediscovery, note that some results we need here were proved in \cite{czgschjh} and the subsequent \init{G.~}Cz\'edli and \init{E.\,T.~}Schmidt~\cite{czgschperm}. 
In \cite{czgschperm}, there are three equivalent definitions for Jordan-H\"older permutations.  Here we combine the treatment given in \cite{czgschjh} and \cite{czgschperm} with Lemma~\ref{trajlemma}. As opposed to \init{H.~}Abels~\cite{abels}, we always assume that $L$ has no cover-preserving diamond.  

\begin{definition}\label{deFperm}
Let $L$ be a semimodular lattice  of length $n$, and assume that $L$ has no cover-preserving diamond. 
Let $C=\set{0=c_0\prec c_1\prec \dots \prec c_n=1}$ and $D=\set{0=d_0\prec d_1\prec \dots \prec d_n=1}$ be maximal chains of $L$. Then the \emph{Jordan-H\"older permutation}  $\perm CD\colon\nset \to \nset $ is  defined by $\perm CD(i)=j$ if{f}
$[c_{i-1},c_i]$ and $[d_{j-1},d_j]$ belong to the same trajectory of $L$.
\end{definition}

It is clear from Lemma~\ref{trajlemma} that $\perm CD$ above is a permutation. The next lemma shows that our definition of $\perm CD$ is the same as that of \init{H.~}Abels~\cite[(3.1)]{abels}, provided $L$ has  no cover-preserving diamond. 

\begin{lemma}\label{ujHpJpw} Let $L$, $C$,  and $D$ be as in Definition~\ref{deFperm}. Then, for $i\in\nset $,  
\[\perm CD(i)=\min\set{j: c_{i-1}\vee d_j=c_i\vee d_j}\text.
\]
The permutation  $\perm CD$ equals the identity permutation $\id$ if{f} $\,C=D$. We have $c_{i-1}\vee d_j=c_i\vee d_j$ if{f} $\,\perm CD(i)\leq j$.  
\end{lemma}

\begin{proof}If follows from Lemma~\ref{trajrestlem} that $\perm CD(i)$ can be computed in $\jgen{C\cup D}$. Hence the first part of the statement can be extracted from  \init{G.~}Cz\'edli and \init{E.\,T.~}Schmidt~\cite[Definition 2.5]{czgschperm} or, more easily, from  \init{G.~}Cz\'edli, \init{L.~}Ozsv\'art, and \init{B.~}Udvari~\cite[Definition 2.4]{czgolub}. Now that we know that our $\perm CD$ is the same as defined in \init{H.~}Abels~\cite{abels}, the middle part follows from \cite[3.5.(a)]{abels}. It also follows from \cite[Theorem 3.3]{czgschperm} or \cite[Lemma 7.2]{czgolub}. The last part is a trivial consequence of the first part.
\end{proof}

While the following lemma needs a proof in \init{H.~}Abels~\cite[Theorem 3.9(f)]{abels}, Definition~\ref{deFperm} in our setting makes it obvious. We compose permutations and maps from right to left, that is, $(f\circ g)(x)= f(g(x))$. 

\begin{lemma}\label{permcat}
Let $L$ be a semimodular lattice of finite length and without cover-preserving diamonds, and let $C$, $D$, and $E$ be maximal chains of $L$. Then the following hold.
\begin{enumeratei}
\item\label{permcata} $\perm CC=\id$, the identity map. 
\item\label{permcatb} $\perm CD=\perm DC ^{-1}$
\item\label{permcatc}$\perm DE\circ\perm CD=\perm CE$.
\end{enumeratei}
\end{lemma}
Equivalently, the lemma above asserts that the maximal chains of $L$ with singleton hom-sets hom$\textup(C,D)=\set{\perm CD}$ form a category, namely, a groupoid. We do not recall further details since although this category (equipped with the weak Bruhat order) determines $L$ by \init{D.\,S.~}Herscovici~\cite{herscovici},
it is rather large and complicated for our purposes. The best way for coordinatization is offered by Lattice Theory.

To give a short illustration of the strength of Lemma~\ref{permcat}, we prove the following corollary even if it is known; see  Remark~\ref{reMrkdh}.

\begin{corollary}\label{idsehGH} If $L$ is a semimodular lattice of finite length and $L$ has no cover-preserving diamonds, then $L$ is finite.
\end{corollary}

\begin{proof} Pick a maximal chain $C$ in $L$. 
By Lemma~\ref{permcat},  if $\perm CD=\perm CE$, then $\id=\perm CD \circ \perm CD^{-1}=\perm CE \circ\perm DC=\perm DE$, and Lemma~\ref{ujHpJpw} implies 
 $D=E$. Hence, $L$ has only finitely many maximal chains. Thus $L$ is finite.
\end{proof}

\section{The main result}\label{secmainres}
We always assume that $n$ belongs to $\mathbb N=\set{1,2,3,\dots}$ and, unless otherwise specified, $k\in\set{2,3,\ldots}$. For a structure $\alg A$, the class of structures isomorphic to $\alg A$ is denoted by $\isomcl{\alg A}$. 
As usual, $S_n$ stands for the group of permutations of the set $\nset $. Consider the class 
\begin{equation*}
\begin{aligned}
\Lat nk =\{&\isomcl \langle L;C_1,\ldots,C_k\rangle : L\text{  is a join-distributive lattice of length }n\text{,  and }\cr
&C_1,\dots,C_k  \text{ are maximal chains of }L\text{ such that  }
\Jir L\subseteq C_1\cup\dots C_k  \}\text.
\end{aligned}
\end{equation*}
 We define a map 
\begin{equation}
\maplp\colon \Lat nk\to S_n^{k-1}
\,\text{ by }\, \isomcl\langle L;C_1,\ldots,C_k\rangle \mapsto \langle \perm{C_1}{C_2}, \perm{C_1}{C_3},\dots, \perm{C_1}{C_k}\rangle\text.
\end{equation}
If $\vpi=\langle  \pi_{2},\dots,\pi_{k}\rangle\in S_n^{k-1}$, then by the corresponding \emph{extended vector} we mean the
$k^2$-tuple 
$
\langle \pi_{ij}: i,j\in\kset  \rangle ,
$
where $\pi_{ij}=\pi_{1j}\circ\pi_{1i}^{-1}$. In general, 
\begin{equation}\label{kiInmd}
\text{$\pi_{ij}$ is always understood as $\pi_{1j}\circ\pi_{1i}^{-1}$,}
\end{equation}
even if  this is not emphasized all the time. 

\begin{definition}\label{defeligl}
By an \emph{eligible $\vpi$-tuple} we mean a $k$-tuple $\vec x=\langle x_1,\ldots, x_k\rangle \in\set{0,1,\ldots,n}^k$ such that $\pi_{ij}(x_i+1 )\geq x_j+1$ holds for all $i,j\in \kset $ such that $x_i<n$. The set of eligible $\vpi$-tuples is denoted by $\elpi$. 
It is a poset with respect to the componentwise order: $\vec x\leq \vec y$ means that, for all $i\in \kset $,  $x_i\leq y_i$.
\end{definition}

For $i\in \kset $, an eligible $\vpi$-tuple $\vec x$ is \emph{initial in its $i$-th component} if for all $\vec y\in\elpi$, 
$x_i=y_i$ implies $\vec x\leq \vec y$. Let $\lancpi i$ be the set of all eligible $\vpi$-tuples that are initial in their $i$-th component. Now we are in the position to define a map
\begin{equation}
\mappl \colon S_n^{k-1}\to  \cdf nk
\,\text{ by }\, \vpi\mapsto \isomcl\langle \elpi; \lancpi 1,\dots,\lancpi k  \rangle\text.
\end{equation}

It is not obvious that $\mappl(\vpi)\in \cdf nk$, but we will prove it soon. 
Now, we can formulate our main result as follows.

\begin{theorem}\label{thmmain}
The maps $\maplp$ and $\mappl$ are reciprocal bijections between 
$\Lat nk$ and $S_n^{k-1}$.
\end{theorem}

This theorem gives the desired coordinatization since each 
$\isomcl\langle L;C_1,\ldots,C_k\rangle$ from $\Lat nk$ is described by its $\maplp$-image, that is, by $k-1$ permutations. The elements of $\langle L;C_1,\ldots,C_k\rangle$ correspond to $k$-tuples over $\set{0,\ldots,n}$, and the $k-1$ permutations  specify which $k$-tuples occur.

\begin{problem}
It would be desirable to characterize those pairs
 $\pair{\vpi}{\vsigma}$ 
 of $(k-1)$-tuples of permutations of $S_n$ for which the lattice part of  $\mappl(\vpi)$ (without the $k$ chains) coincides with that of $\mappl(\vsigma)$.  
However, in spite of the theory developed in  \init{D.\,S.~}Herscovici~\cite{herscovici},  we do not expect an elegant characterization. While the requested characterization for $k=2$ is known from  \init{G.~}Cz\'edli and \init{E.\,T.~}Schmidt~\cite{czgschperm} and it was used in \init{G.~}Cz\'edli, \init{L.~}Ozsv\'art, and \init{B.~}Udvari~\cite{czgolub},
even the case $k=3$ seems to be quite  complicated; this is witnessed by the following example.
\end{problem}
\begin{example}\label{exMple} Let $L=\set{0,a,b,c,a\vee b, a\vee c, b\vee c,1}$ be the 8-element boolean lattice. Consider the maximal chains $C_1=\set{0,a,a\vee b, 1}$,  $C_2=\set{0,b,a\vee b, 1}$, $C_3=\set{0,c,b\vee c, 1}$, and also the maximal chains $C_1'=\set{0,a,a\vee b, 1}$, $C_2'=\set{0,b,b\vee c, 1}$, and $C_3'=\set{0,c,a\vee c, 1}$. Then
\[\begin{aligned}
\perm{C_1}{C_2}&=\begin{pmatrix}1&2&3\cr 2&1&3\end{pmatrix},\qquad
\perm{C_1}{C_3}=\begin{pmatrix}1&2&3\cr 3&2&1\end{pmatrix},\cr
\perm{C_1'}{C_2'}&=\begin{pmatrix}1&2&3\cr 3&1&2\end{pmatrix},\qquad
\perm{C_1'}{C_3'}=\begin{pmatrix}1&2&3\cr 2&3&1\end{pmatrix}\text.
\end{aligned}
\] 
Let $\vpi=\pair{\perm{C_1}{C_2}}{\perm{C_1}{C_3}}$ and $\vsigma=\pair{\perm{C_1'}{C_2'}}{\perm{C_1'}{C_3'}}$. These two vectors  look very different  since $\vpi$ consists of transpositions while $\vsigma$ does not; furthermore, the second component of $\vsigma$ is the inverse of the first component, while this is not so for $\vpi$. However, by construction and Theorem~\ref{thmmain}, $\vpi$ and $\vsigma$ determine the same lattice $L$.
\end{example}

\section{Proving the main result}\label{sectproofs}
Some of the auxiliary statements we will prove are valid under a seemingly weaker assumption than requiring $\isomcl L\in \Lat nk$. Therefore, in accordance with Remark~\ref{reMrkdh}, we prove a (seemingly) stronger statement with  little extra effort. To do so, consider the class 
\begin{equation*}
\begin{aligned}
\cdf nk =\{\isomcl \langle L;C_1,\ldots,C_k\rangle : L\text{ is semimodular of length } n\text{, }L \text{ contains}&\cr
\text{no cover-preserving diamond,   and }  C_1,\ldots, C_k \text { are maximal chains of $L$.}\}& 
\end{aligned}
\end{equation*}
This notation comes from  ``Cover-preserving Diamond-Free''. 
Although we know from Proposition~\ref{Dlwltcs} and Remark~\ref{reMrkdh} that  
$\cdf nk$ equals $\Lat nk$, we will give a new  proof for  this equality. We will use only the obvious  $\Lat nk \subseteq \cdf nk$. Keeping the notation of the original map, we extend its range as follows:  
\begin{align*}
\maplp\colon \cdf nk\to S_n^{k-1}
\text{, where }\, \isomcl\langle L;C_1,\ldots,C_k\rangle \mapsto \langle \perm{C_1}{C_2}, \perm{C_1}{C_3},\dots, \perm{C_1}{C_k}\rangle\text.
\end{align*}
Our aim in this section is to prove the following statement, which implies Theorem~\ref{thmmain} and harmonizes with Remark~\ref{reMrkdh}.

\begin{proposition}\label{almostmain}\ 

$\textup{(A)}$ $\cdf nk=\Lat nk$.

$\textup{(B)}$ 
The maps $\maplp$ and $\mappl$ are reciprocal bijections. 
\end{proposition}

The proof of Proposition~\ref{almostmain} will need the following lemma.

\begin{lemma}\label{siodsGz} If $\vpi\in S_n^{k-1}$, then $\mappl(\vpi)\in \Lat nk$.
\end{lemma}

%

\begin{proof} Convention \eqref{kiInmd} should be kept in mind. 
Clearly, $\langle n,\ldots,n\rangle\in\elpi$. Hence $\elpi$ has a top element. Let $\vec x=\langle x_1,\ldots,x_k\rangle$ and $\vec y=\langle y_1,\ldots,y_k\rangle$ belong to $\elpi$, and let $z_i=x_i\wedge y_i=\min\set{x_i,y_i}$, for $i\in\kset $.
Assume that $j\in\kset $ such that $z_j<n$. Since $\vec x,\vec y\in\elpi$, we have $\pi_{jt}(x_j+1)\geq x_t+1$ and $\pi_{jt}(y_j+1)\geq y_t+1$, for every $t\in\kset $. Let, say, $z_j=x_j$. We obtain 
$\pi_{jt}(z_j+1) = \pi_{jt}(x_j+1)\geq x_t+1 \geq z_t+1$. Hence $\vec z=\langle z_1,\ldots,z_k\rangle\in\elpi$, and we conclude that $\elpi$ is a lattice.

If we had $\vec x=\tuple{ n,x_2,\ldots,x_k}\in\elpi$ and $x_i\neq n$ for some $i\in\set{2,\ldots,k}$, then $\pi_{i1}(x_i+1)\leq n < x_1+1$ would contradict $\vec x\in\elpi$. Hence $\langle n,\ldots,n\rangle$ is the only vector in $\elpi$ with first component $n$, and we conclude that this vector belongs to $\lancpi 1$. Next, assume that $\vec y=\langle y_1,\ldots,y_k\rangle$ and 
$\vec z=\langle z_1,\ldots,z_k\rangle$ belong to $\lancpi 1$ such that $y_1\leq z_1$. Since $\elpi$ is meet-closed, 
$\vec y\wedge \vec z=\langle y_1,y_2\wedge z_2,\ldots,y_k\wedge z_k\rangle\in\elpi$. Since $\vec y$ is initial in its first component, 
$\vec y\leq \vec y\wedge \vec z$, which gives $\vec y\leq \vec z$. Thus $\lancpi 1$ is a chain.

A vector $\vec x$ in $\nset ^k$ is a \emph{$\vpi$-orbit}, if $\pi_{ij}(x_i)=x_j$ for all $i,j\in\kset$.
Equivalently, if it is of the form 
\begin{equation}\label{slVCbk} \vec x=\tuple{ \pi_{i1}(b),\dots,\pi_{ik}(b)} \rangle
\end{equation}
for some $i\in\kset $ and $b\in \nset $.   A $\vpi$-orbit need not belong to $\elpi$.
A vector $\vec y=\langle y_1,\ldots,y_k \rangle$ in $\set{0,\ldots,n-1}^k$ is  \emph{ suborbital with respect to $\vpi $} if 
\begin{equation}\label{suboRbital}
\text{
$\langle y_1+1,\ldots, y_k+1\rangle$ is a $\vpi$-orbit.}
\end{equation} 
Clearly, suborbital vectors belong to $\elpi$. For each $b\in \set{0,\dots,k-1}$,  \eqref{slVCbk}, applied for $i=1$, shows that  there exists a vector $\vec x\in \elpi$ whose first component is $b$. Let $\vec y$ be the meet of all these vectors. Clearly, $y_1=b$ and $y\in \lancpi 1$. This shows that $\lancpi 1$ is a chain of length $n$, and so are the $\lancpi i$ for $i\in\kset $.

Next, let $B(\vpi)$ denote the set of suborbital vectors with respect to $\vpi$.  
We have $B(\vpi)\subseteq \elpi$. 
By \eqref{slVCbk}, each $b\in\nset $ is the $i$-th component of exactly one $\vpi$-orbit. Therefore, 
\begin{equation}\label{dkcBGb}
|B(\vpi)|=n\text.
\end{equation}
We assert that 
\begin{equation}\label{msBxG}
\begin{aligned}
&\text{for each $\vec x\in\elpi$, there is a unique $U\subseteq B(\vpi)$ such}\cr
&\text{that 
$\vec x=\bigwedge U$ is an irredundant meet decomposition.}
\end{aligned}
\end{equation}
Let $\vec x=\tupk x\in \elpi\setminus\set{\tuple{n,\ldots,n}}$, and let $I=\set{i: x_i<n}$. 
 For $i\in I$, let 
\[\vec y^{\,(i)}=\tuple{\pi_{i1}(x_i+1)-1,\dots,\pi_{ik}(x_i+1)-1}\text.
\]
It belongs to $B(\vpi)$, whence  $\vec y^{\,(i)}\in\elpi$. Since $\vec x$ is also in $\elpi$, for $j\in\kset $ we have
$x_j=x_j+1-1\leq \pi_{ij}(x_i+1)-1= y_j^{(i)}$. Hence $\vec x\leq \vec y^{(i)}$, and we conclude 
$\vec x\leq \bigwedge\set{ \vec y^{\,(i)}: i\in I }$. The converse inequality also holds since $x_i=n$ for $i\notin I$ and $y_i^{(i)}=\pi_{ii}(x_i+1)-1 = x_i$ for $i\in I$. That is, $\vec x = \bigwedge\set{ \vec y^{\,(i)}: i\in I }$. Since the meetands here are not necessarily distinct, we let 
\[\text{
$J=\{j\in I:\vec y^{\,(i)}\neq  \vec y^{\,(j)}$ for all $i<j$ such
that $i\in I\}$.
}\] 
Clearly, $\vec x = \bigwedge\set{ \vec y^{\,(j)}: j\in J }$.
We assert that if  $V\subseteq B(\vpi)$ and  $\vec x=\bigwedge V$, then 
$\set{\vec y^{\,(j)}: j\in J}\subseteq V$.  To show this, let  $j\in J$. Since the $j$-th components of the vectors of $V$ form a chain and the meet of these components equals $x_j$, there exists a $\vec u\in V$ such that $u_j=x_j$. The $j$-th component of $\vec y^{\,(j)}$ is also $x_j$ by definition. Since any two orbital vectors with a common $j$-th component are equal, we obtain 
$\vec y^{\,(j)}=\vec u\in V$. Thus we obtain $\set{\vec y^{\,(j)}: j\in J}\subseteq V$. Since the $\vec y^{\,(j)}$, for $j\in J$,  are pairwise distinct, we can take $U=\set{\vec y^{\,(j)}: j\in J}$, and we conclude \eqref{msBxG}.

Now, let $\vec a\in\Mir{(\elpi)}$. Since $\vec a= \ibvec 1\wedge\dots \wedge \ibvec s$ for appropriate $ \ibvec 1,\ldots, \ibvec s$ in $B(\vpi)$ by \eqref{msBxG}, we obtain $s=1$ and $\vec a= \ibvec 1\in B(\vpi)$. That is, $\Mir{(\elpi)}\subseteq B(\vpi)$. For the sake of contradiction, suppose $B(\vpi)\setminus\Mir{(\elpi)}\neq \emptyset$, and let $\vec b\in B(\vpi)\setminus\Mir{(\elpi)}$. 
Observe that $U'=\set{ \vec b}$ gives an irredundant meet-decomposition according to \eqref{msBxG}. 
There are a minimal $t\in\mathbb N$ and $\iwvec 1,\ldots,\iwvec t\in\Mir{(\elpi)}$ such that $\vec b=\iwvec 1\wedge\dots\wedge \iwvec t$. This meet is irredundant by the minimality of $t$, and $t\geq 2$ since $\vec b\neq\tuple{n,\ldots,n}$ and $\vec b\notin \Mir{(\elpi)}$. However, $\iwvec 1,\ldots,\iwvec t\in B(\vpi)$ since $\Mir{(\elpi)}\subseteq B(\vpi)$. Thus $U''=\set{\iwvec 1,\ldots,\iwvec t}$ also gives an irredundant meet-decomposition according to \eqref{msBxG}. This is a contradiction since $U'\neq U''$. Hence,
\begin{equation}\label{mlpbP}
\Mir{(\elpi)} = B(\vpi),
\end{equation}
and $\elpi$ has unique meet-irreducible decompositions. Thus we conclude from Proposition~\ref{Dlwltcs} that 
 $\elpi$ is join-distributive.

Next, we have to show that $\lancpi i$ is a maximal chain for $i\in\kset $. Since $|\lancpi i|=n+1$, it suffices to show that $\elpi$ is of length at most $n$. To prove this by contradiction, suppose  $H=\set{h_0\prec h_1\prec \dots\prec h_{n+1}}$ is a chain of $\elpi$. Let $W_i=\Mir{(\elpi)}\cap \filter{h_i}$. Clearly, $W_0 \supsetneq W_1 \supsetneq \dots \supsetneq W_{n+1}$, which contradicts the fact that $|\Mir{\elpi}|=|B(\vpi)|=n$ by \eqref{dkcBGb} and \eqref{mlpbP}.

Finally, let $\vec x=\tupk x\in \elpi$. For $i\in\kset$, let $\vec y^{\,(i)}\in\elpi$ be the smallest vector whose $i$-th component is $x_i$. Clearly, $\vec y^{\,(i)}\in \lancpi i$, $\vec y^{\,(i)}\leq \vec x$, and $\vec y^{\,(1)}\vee\dots\vee \vec y^{\,(k)}=\vec x$. Hence $\Jir{(\elpi)}\subseteq \lancpi 1\cup\dots\cup \lancpi k$.
\end{proof}

\begin{lemma}[Roof Lemma]\label{rooflemma} If  $k\in \mathbb N$, $\isomcl{\alg L}$ belongs to $\cdf nk$,  $x_1,\ldots,x_k\in L$,  $x_1\vee \dots\vee x_k=1$, and $\inp_i\in\print{\filter{x_i}}$ for $i\in\kset$, then $\inp_1,\ldots,\inp_k$ cannot belong to the same trajectory of $L$. 
\end{lemma}

\begin{proof}
We can assume that $1\notin \set{x_1,\ldots,x_k}$ since otherwise $\print{\filter{x_i}}=\emptyset$ for some $i\in\kset$, and the statement trivially holds. Therefore, we can also assume that $k>1$.
 
First, we deal with $k=2$. For the sake of contradiction, suppose $\inp_1=[u_1,v_1]$ and $\inp_2=[u_2,v_2]$ belong to the same trajectory $T$. 
For $i\in\set{1,2}$, pick a maximal chain $X_i$ containing $x_i$ such that  $\inp_i\in\print {X_i}$. Let $M=\jgen{X_1\cup X_2}$, the join-subsemilattice generated by $X_1\cup X_2$. 
We know from \init{G.~}Cz\'edli and \init{E.\,T.~}Schmidt~\cite[Lemma 2.4]{czgschjh} that $M$ is a cover-preserving 0,1-join-subsemilattice of $L$, and it is a slim semimodular lattice. Its length, $n$, is the same as that of $L$.
We obtain from Lemma~\ref{trajrestlem} that $\inp_1$ and $\inp_2$ belong to the same trajectory $T$ of $M$.
By \cite[Lemma 2.2]{czgschjh}, $M$ is a planar lattice, and it has a planar diagram whose left boundary chain is $X_1$.  The trajectories of planar semimodular lattices join-generated by two chains are well-understood. By \cite[Lemma 2.9]{czgschjh}, there is an interval $\inp_3=[u_3,v_3]$ such that both $\inp_1$ and $\inp_2$ are  up-perspective to $\inp_3$.
This means that $u_i=v_i\wedge u_3$ and $v_3=v_i\vee u_3$, for $i\in\set{1,2}$. Hence $1=x_1\vee x_2\leq u_1\vee u_2\leq u_3<v_3\leq 1$, which is a contradiction that proves the statement for  $k=2$.

Next, to proceed by induction, assume that $k>2$ and the lemma holds for smaller values. To obtain a contradiction, suppose that $\inp_1,\ldots,\inp_k$ belong to the same trajectory $T$ of $L$. 
Let $y_1=x_2\vee\cdots\vee x_{k}$. We can assume $y_1\neq 1$ since otherwise $x_1$ can be omitted and the induction hypothesis applies. Pick a maximal chain $U_1$ of $L$ such that $y_1\in U_1$. 
By Lemma~\ref{trajlemma}, there is a unique $\inq_1\in T\cap\print{U_1}$. We assert that $\inq_1\in\print{\filter {y_1}}$. Suppose not, and let $R =T\cap \print{\ideal{y_1}}$. It is nonempty by Lemma~\ref{trajlemma}. Hence, it is a trajectory of  $\ideal{y_1}$ by Lemma~\ref{trajrestlem}. Thus the induction hypothesis, together with Lemma~\ref{trajlemma}, yields an $i\in\set{2,\ldots,k}$ and a prime interval $\inr_i$  such that $\inr_i\in  \print{\ideal{x_i}}\cap R  \subseteq \print{\ideal{x_i}}\cap T$. Since we can clearly take a maximal chain $V$ of $L$ through $x_i$  such that $\inp_i,\inr_i\in\print{V}$, we obtain a contradiction by  Lemma~\ref{trajlemma}. Thus we conclude $\inq_1\in\print{\filter{y_1}}$.

Since $1\in\kset$ in the argument above does not play any special role, we obtain that $T$ contains a prime interval $\inq_2$ in the filter $\filter{y_2}$ generated by $y_2=x_1\vee x_3\vee\dots\vee x_k$. This contradicts the induction hypothesis since  $y_1\vee y_2=1$.
\end{proof}

Let $\alg L=\tuple{L;C_1,\ldots,C_k}\in \cdf nk$. We denote the elements of $C_i$ as follows:
\[ 
C_i=\set{0=c_0^{\,(i)} \prec c_1^{\,(i)} \prec \dots \prec c_n^{\,(i)} =1 }\text.
\]
A vector
$\vec x\in \nnset^k$ is called \emph{$\alg L$-maximal} if for all $i\in \kset$, we have 
\[ c_{x_1}^{\,(1)}\vee \cdots\vee c_{x_k}^{\,(k)} <  
c_{x_1}^{\,(1)}\vee \cdots\vee c_{x_{i-1}}^{\,(i-1)}\vee c_{x_i+1}^{\,(i)}\vee  c_{x_{i+1}}^{\,(i+1)}\vee \cdots\vee c_{x_k}^{\,(k)}\text.
\]

\begin{lemma}\label{maxEliglma} Let $\isomcl\alg L=\isomcl\tuple{L;C_1,\ldots,C_k}\in \cdf nk$. Let $\vpi$ denote its $\xi$-image, and let $\vec x\in \nnset^k$. Then $\vec x$ is $\alg L$-maximal if{f} it is an eligible $\vpi$-tuple.
\end{lemma}

\begin{proof} 
To prove the ``only if'' part, let $\vec x=\tupk x$ be an $\alg L$-maximal vector, and let $i\neq j\in\kset$. By $\alg L$-maximality, 
\begin{align*}
  c_{x_i}^{\,(i)}  \vee c_{x_j}^{\,(j)}&\vee\bigvee\bigset{ c_{x_t}^{\,(t)}: t\in\kset\setminus\set{i,j} } \cr
< {}  
&c_{x_i+1}^{\,(i)}\vee c_{x_j}^{\,(j)}\vee\bigvee\bigset{ c_{x_t}^{\,(t)}: t\in\kset\setminus\set{i,j} } \text.
\end{align*}
Hence $c_{x_i}^{\,(i)}  \vee c_{x_j}^{\,(j)}  <  c_{x_i+1}^{\,(i)}\vee c_{x_j}^{\,(j)}$. Thus,  by Lemma~\ref{ujHpJpw}, 
$\perm{C_i}{C_j}(x_i+1) > x_j$. Therefore, $\vec x$ is  an eligible $\vpi$-tuple.

To prove the converse implication by contradiction, suppose that $\vec x=\tupk x$ is an eligible $\vpi$-tuple but it is not $\alg L$-maximal. Let, say, the $k$-th component of $\vec x$ violate $\alg L$-maximality, and let 
\begin{equation}\label{sWodBjT}
\text{
$u=c_{x_1}^{\,(1)}\vee\dots\vee c_{x_{k-1}}^{\,(k-1)}$.}
\end{equation} 
We have $u\vee c_{x_k}^{\,(k)}=u\vee c_{x_k+1}^{\,(k)}$.  For $i\in \set{1,\ldots,k-1}$, extend $(C_i\cap \ideal {c_{x_i}^{\,(i)}})\cup\set u$ to a maximal chain $U_i$ of $L$. We denote the trajectory of $L$ 
that contains $[c_{x_k}^{\,(k)}, c_{x_k+1}^{\,(k)}]$ by $T$. 
Since $u\vee c_{x_k}^{\,(k)}=u\vee c_{x_k+1}^{\,(k)}$, Lemma~\ref{ujHpJpw} yields  
$\perm {C_k}{U_i}(x_k+1)\leq h(u)$, where $h$ is the height function. Hence, by Definition~\ref{deFperm}, there is a $\inp_i\in\print{U_i}\cap T$ such that $\inp_i$ is below $u$, that is, 
\begin{equation}\label{nsiseoe}
\inp_i\in \print{U_i\cap \ideal u}\cap T \text.
\end{equation}
On the other hand, the eligibility of $\vec x$ gives $\perm{C_k}{C_i}(x_k+1)\geq x_i+1$. Hence, again  by Definition~\ref{deFperm},  $C_i$ contains a prime interval of $T$ above $c_{x_i}^{\,(i)}$. 
Therefore, by Lemma~\ref{trajlemma}, $C_i$ does not contain any prime interval of $T$ below $c_{x_i}^{\,(i)}$. But $C_i\cap \ideal c_{x_i}^{\,(i)}= U_i\cap \ideal c_{x_i}^{\,(i)}$, and we conclude that $  {\print{U_i\cap \ideal c_{x_i}^{\,(i)}}}\cap T=\emptyset$. 
Combining this with \eqref{nsiseoe}, we obtain that $T\cap\print{\ideal u}$, which is a trajectory of $\ideal u$ by Lemma~\ref{trajrestlem}, contains a prime interval  in $\filter c_{x_i}^{\,(i)}$, for 
$i\in\set{1,\ldots,k-1}$. This, together with \eqref{sWodBjT}, contradicts  Lemma~\ref{rooflemma}.
\end{proof}

The following lemma generalizes \init{G.~}Cz\'edli and \init{E.\,T.~}Schmidt~\cite[Lemma 2.3]{czgschperm}. We use the notation preceding Lemma~\ref{maxEliglma}. The set of suborbital vectors is still denoted by $B(\vpi)$, as above \eqref{dkcBGb}. 
For $u\in L$, the \emph{foot} of $u$ is the following vector in $\nnset^k$: 
\[\feet u=\tuple{\foot u1,\ldots,\foot uk}=\Bigl\langle 
\max\set{j: c_{j}^{\,(1)} \leq u },
\dots,
\max\set{j: c_{j}^{\,(k)} \leq u }  
\Bigr\rangle\text.
\]

\begin{lemma}\label{ujsuborblMa}
If $\,\isomcl\alg L\in\cdf nk$, then 
$\set{\feet u: u\in\Mir L} \subseteq  B(\vpi)$, where $\vpi=\maplp(\isomcl{\alg L})$. If $\,\isomcl\alg L\in\Lat nk$, then even  $\set{\feet u: u\in\Mir L} =  B(\vpi)$ holds.
\end{lemma}

\begin{proof}
For $k=2$, the lemma is only a reformulation of \cite[Lemma 2.3]{czgschperm}. Namely, in this case, one takes the trajectory  $[u,\upstar u]$; it contains a unique prime interval $[a_0,a_1]$, and $\foot u1=h(a_0)$; and analogously for $\foot u2$.

Hence, we assume $k>2$.  Let $u\in \Mir L$. We are going to prove  
\begin{equation}\label{scdsie}
\pi_{12}(\foot u1 +1)=\foot u2+1\text.
\end{equation} 
Let $K=\jgen{C_1\cup C_2}$; it is a slim, semimodular, cover-preserving join-subsemi\-lattice of $L$ by \init{G.~}Cz\'edli and \init{E.\,T.~}Schmidt~\cite[Lemma 2.4]{czgschjh}.
Let $u_0$ be the largest element of $K\cap \ideal u$. 

First, if $u_0\in \Mir K$, in particular if $u=u_0$,  then \eqref{scdsie} follows from \cite[Lemma 2.3]{czgschperm}, Lemma~\ref{trajrestlem}, and the argument detailed in the first paragraph of the present proof. 

Second, for the sake of contradiction, suppose that 
\begin{equation}\label{suctNg}
\text{$u_0$ is distinct from $u$ and  $u_0$ is meet-reducible in  $K$.}
\end{equation}
Pick two distinct covers, $a_0$ and $b_0$ of $u_0$ in $K$. 
Let $u_0\prec u_1\prec\dots\prec u_t=u$ be a maximal chain in the interval $[u_0,u]$. The unique cover of $u$ in $L$ is denoted by $\ucsil$. By semimodularity, $u\vee a_0$ and $u\vee b_0$ cover $u$. Hence, $u\vee a_0 = u\vee b_0=\upstar u$, which implies $\set{a_0,b_0}\subseteq \ideal{\upstar u}$.
Since $u_1$ is not in $K$, the elements $a_0,b_0, u_1$ are three distinct atoms in the filter $\filter{u_0}$. If they are not independent, then we can select two of them that are independent in the sense of \init{G.~}Gr\"atzer~\cite[Theorem 380]{GGLT},
and we easily obtain that  $a_0,b_0, u_1$ generate a cover-preserving diamond, which is a contradiction. 
Hence they are independent, and they  generate a cover-preserving boolean sublattice, a \emph{cube} for being brief. Define $a_1=a_0\vee u_1$ and $b_1=b_0\vee u_1$; they belong to $\ideal \ucsil$ since so are $a_0,b_0,u_1$.  Then the cube we have just obtained is 
$\set{u_0,a_0,b_0, a_0\vee b_0, u_1,a_1,b_1, a_1\vee b_1}$, and it is in $\ideal{\upstar u}$. Since $u_1\prec a_1$ and $u_1\prec b_1$, $u_1\neq u\in\Mir L$.

Now, we repeat the procedure within $[{u_1}, \upstar u]$ instead of $[{u_1}, \upstar u]$. If we had, say, $u_2=a_1$, then 
$a_0<a_1=u_2\leq u$ and $a_0\in K$ would contradict the definition of $u_0$. Hence $a_1,b_1,u_2$ are distinct covers of $u_1$. As before, they generate a cube, which is 
$\{u_1, a_1,b_1, u_2,a_1\vee b_1,
a_2=a_1\vee u_2, b_2=b_1\vee u_2, a_2\vee b_2\}$. Since $u_2\prec a_2$ and $u_2\prec b_2$, $u_2\neq u\in\Mir L$. 

And so on. After $t$ steps, we obtain 
$u=u_t\notin\Mir L$, a contradiction. This proves \eqref{scdsie}. We obtain $\pi_{ij}(\foot ui +1)=\foot uj+1$ similarly, and we conclude $\feet u\in B(\vpi)$. This proves the first part of the lemma.

For $u\in\Mir L$, we have $u= c_{\foot u1}^{\,(1)} \vee\dots\vee  c_{\foot uk}^{\,(k)}$ since  $\Jir L\subseteq C_1\cup\dots\cup C_k$. Hence, $\feet u$ determines $u$, 
the map $\Mir L\to \set{\feet u: u\in\Mir L}$, defined by $u\mapsto \feet u$, is a bijection, and $|\set{\feet u: u\in\Mir L}| =|\Mir L|$. Thus
the second part of the lemma follows from the first part, Proposition~\ref{Dlwltcs}\eqref{Dlwltcse}, and \eqref{dkcBGb}. 
\end{proof}

\begin{remark} Since the proof above excludes  \eqref{suctNg}, we conclude that if $u\in\Mir L$, then  
$u=c_{\foot ui}^{\,(i)} \vee  c_{\foot uj}^{\,(j)}$, for all $i\neq j\in\kset$.
\end{remark}

\begin{proof}[Proof of Proposition~\ref{almostmain}] 
Clearly,  $\Lat nk\subseteq \cdf nk$. Hence we obtain  from Lemma~\ref{siodsGz} that $\mappl$ is a map from $S_n^{k-1}$ to $\cdf nk$. 

Let $\isomcl{\alg L}=\isomcl \tuple{L;C_1,\ldots,C_k}\in \cdf nk$. 
Denote $\maplp(\isomcl\alg L)$ by $\vpi$. This makes sense since, clearly, every $\alg L'$ isomorphic to $\alg L$ gives the same $\vpi$.  It is obvious  by Section~\ref{JHsection} that $\vpi\in S_n^{k-1}$.  That is, $\maplp$ is a map from $\cdf nk$ to $S_n^{k-1}$.

Now, for  $\isomcl \alg L\in \cdf nk$ and $\pi=\maplp(\isomcl\alg L)$ above, we  use the notation introduced before Lemma~\ref{maxEliglma}. Since $\elpi$ coincides with the set of $\alg L$-maximal vectors by Lemma~\ref{maxEliglma}, the map $\mu\colon L\to L(\vpi)$, defined by $u\mapsto\feet u$, 
is an order-isomorphism. Thus $\mu$ is a lattice isomorphism. To deal with $\mu(C_i)$, let $x\in C_i$. For any $y\in L$, if $\foot yi=h(x)=\foot xi$, then $x\leq y$ and, hence, $\mu(x)=\feet x\leq \feet y=\mu(y)$. This shows that $\mu(x)$ is  initial in its $i$-th component. Thus $\mu(C_i)=C_i(\vpi)$, for $i\in \kset$. Hence, $\mappl(\vpi)=\isomcl{\alg L}$.  This shows that 
\begin{equation}\label{lsBkcD}
\text{$\mappl \circ\maplp$ is the identity map on $\cdf nk$.}
\end{equation}

Now, we are in the position to prove $\cdf nk=\Lat nk$. The inclusion $\Lat nk\subseteq \cdf nk$ is trivial. Conversely, let $\isomcl{\alg L}$ belong to $\cdf nk$, and denote $\maplp(\isomcl{\alg L})$ by $\vpi$.  By \eqref{lsBkcD}, $\isomcl{\alg L} = \mappl(\vpi)$. Hence Lemma~\ref{siodsGz} yields $\isomcl{\alg L}\in \Lat nk$. This proves $\cdf nk=\Lat nk$.

Next, let $\vpi\in S_n^{k-1}$, and consider $ \mappl(\vpi)= \alg \elpi=\tuple{\elpi;\lancpi 1,\dots,\lancpi k}$. Let $\vec x\in \elpi$.  Since $\mappl(\vpi)\in \Lat nk$ by Lemma~\ref{siodsGz}, the length of $\lancpi 1$ is $n$.  This, together with the fact that no two distinct vectors in 
$\lancpi 1$ have the same first component,
implies that each $t\in\nnset$ is the first component of exactly one vector in $\lancpi 1$. Thus there is a unique $\vec y\in \lancpi 1$ such that $y_1=x_1$. Since $\vec y$ is initial in its first component, it is the largest vector in $\lancpi 1 \cap \ideal{\vec x}$. Clearly, the height of $\vec y\,$ is  $y_1=x_1$. 
Hence, using the notation given before Lemma~\ref{ujsuborblMa}, $\foot {\vec x}1=x_1$, and similarly for other indices. Therefore, $\feet{\vec x}=\vec x$ holds for all $\vec x\in \elpi$.
Applying this observation to $B(\vpi)=\{$suborbital vectors with respect to $\vpi\}$, we conclude $\set{ \feet u: u\in B(\vpi)}=B(\vpi)$. Hence, by \eqref{mlpbP},  
\begin{equation}\label{dieGh}
\text{$\set{ \feet u: u\in \Mir{(\elpi})}=B(\vpi)$.
}\end{equation}
On the other hand, Lemma~\ref{ujsuborblMa}, 
applied  to $\alg L(\vpi)=\eta(\vpi)$, yields the equality $\set{ \feet u: u\in \Mir{(\elpi})} = B\bigl(\maplp(\mappl(\vpi))\bigr)$. Combining this equality with \eqref{dieGh}, we obtain $B\bigl(   \maplp(\mappl(\vpi))\bigr)=B(\vpi)$. 
This means that  $\vpi$ and $\maplp(\mappl(\vpi))$ have exactly the same 
suborbital vectors. Hence, they have the same orbits. Since they are determined by their orbits, we conclude  $\maplp(\mappl(\vpi))=\vpi$. Thus $\maplp\circ\mappl$ is the identity map on $S_n^{k-1}$.
\end{proof}


\section{Coordinatizing antimatroids and convex geometries}\label{secantimatr}
The concept of antimatroids is due to \init{R.\,E.~}Jamison-Waldner~\cite{jamison}, who was the first to use the term ``antimatroid''. 
At the same time, an equivalent  complementary concept was introduced by  \init{P.\,H.~}Edelman~\cite{edelman} under the name ``anti-exchange closures''. 
There are several ways to define antimatroids, see \init{D.~}Armstrong~\cite[Lemma 2.1]{armstrong}; here we accept the following one. The set of all subsets of a set $E$ is denoted by $P(E)$.
\begin{definition}
A pair $\pair{E}{\alg F}$ is an \emph{antimatroid} if it satisfies the following properties:
\begin{enumeratei}
\item\label{antimatdefa} $E$ is a finite set, and $\emptyset\neq \alg F \subseteq P(E)$;
\item\label{antimatdefb} $\alg F$ is a \emph{feasible set}, that is, for each nonempty $A\in \alg F$, there exists an $x\in A$ such that $A\setminus\set x\in\alg F$;
\item\label{antimatdefc} $\alg F$ is closed under taking unions;
\item\label{antimatdefd} $E=\bigcup\set{A: A\in \alg F}$.  
\end{enumeratei}
\end{definition}

If $\pair{E}{\alg F}$ satisfies \eqref{antimatdefa}, \eqref{antimatdefb}, and \eqref{antimatdefc}, but possibly not \eqref{antimatdefd},  then  the elements of $E\setminus\bigcup\set{A: A\in \alg F}$  are called \emph{dummy points}. Many authors allow dummy points, that is, do not stipulate \eqref{antimatdefd} in the definition of antimatroids. However, this is not an essential difference since a structure $\pair{E}{\alg F}$ satisfying  \eqref{antimatdefa}, \eqref{antimatdefb}, and \eqref{antimatdefc} is clearly characterized by the antimatroid $\pair{\bigcup\set{A: A\in \alg F}}{\alg F}$  (in our sense) and the number $|E\setminus\bigcup\set{A: A\in \alg F}|$ of dummy points. Note that without \eqref{antimatdefd}, the forthcoming Proposition~\ref{pdZT} would fail.

Antimatroids were generalized to more general systems called greedoid by 
\init{B.~}Korte and \init{L.~}Lov\'asz~\cite{kortelovasz81} and \cite{kortelovasz83}. 
A \emph{closure operator} on a set $E$ is an extensive, monotone, and idempotent map $\Phi\colon P(E)\to P(E)$. That is,  $X\subseteq \Phi(X)=\Phi(\Phi(X))\subseteq \Phi(Y)$,   for all $X\subseteq Y\in P(E)$. A \emph{closure system} on $E$ is a nonempty subset of $P(E)$ 
that is closed under taking arbitrary intersections. In particular, a closure system on $E$ always contains the empty intersection, $E$.  There is a well-known bijective correspondence between closure operators and closure systems; see
\init{S.~}Burris and \init{H.\,P.~}Sankappanavar~\cite[I.\S5]{burrissankap}.
The closure system corresponding to a closure operator $\Phi$ consists of the \emph{closed sets}, that is, of the sets $X\in P(E)$ satisfying
$\Phi(X)=X$. The closure operator corresponding to a closure system $\alg C$  is the map $P(E)\to P(E)$, defined by 
$X\to\bigcap\set{Y\in \alg C: X\subseteq Y}$. Now, we define a concept closely related to antimatroids, 
see \init{P.H.~}Edelman~\cite{edelman} and    \init{K.~}Adaricheva, \init{V.\,A.~}Gorbunov, and \init{V.\,I.~}Tumanov~\cite{r:adarichevaetal}. Let us emphasize that \eqref{convgeoc} below is stipulated in \cite{r:adarichevaetal} and also in \init{D.~}Armstrong~\cite{armstrong}.
\begin{definition}\label{convgeo}
A pair $\pair{E}{\Phi}$ is a \emph{convex geometry}, also called \emph{anti-exchange system}, if it satisfies the following properties:
\begin{enumeratei}
\item\label{convgeoa} $E$ is a finite set,  and $\Phi\colon P(E)\to P(E)$ is a closure operator. 
\item\label{convgeob} If $\Phi(A)=A\in P(E)$, $x,y\in E$, $x\notin A$, $y\notin A$, $x\neq y$, and $x\in \Phi(A\cup\set y)$, then $y\notin \Phi(A\cup\set x)$. (This is the so-called \emph{anti-exchange property}.)
\item\label{convgeoc}  $\Phi(\emptyset)=\emptyset$.
\end{enumeratei}
For a closures system $\alg G$ on $E$ with corresponding closure operator $\Phi$, $\pair E{\alg G}$ is a \emph{convex geometry}
if so is  $\pair{E}{\Phi}$ in the above sense. In what follows, the notations $\pair{E}{\Phi}$  and $\pair E{\alg G}$ can be used interchangeably for the same mathematical object. The members of $\alg G$ are called \emph{closed sets}. 
\end{definition}

It follows easily from \eqref{convgeob}  that if $\pair E{\alg G}$ is a convex geometry, then 
\begin{equation}\label{adodkF}
\text{for each $B\in \alg G\setminus\set E$, there is an $x\in E\setminus B$ such that $B\cup\set x\in \alg G$.}
\end{equation}
The  following statement is taken from the book \init{B.~}Korte, \init{L.~}Lov\'asz, and \init{L.~}Schrader~\cite[Theorem III.1.3]{korteatalbook}, see also \eqref{adodkF} together with \init{D.~}Armstrong~\cite[Lemma 2.5]{armstrong}.

\begin{proposition}\label{pdZT} Let $E$ be a finite set, and let $\emptyset\neq \alg F\subseteq P(E)$. Then  $\alg A=\pair E{\alg F}$ is an antimatroid if{f} $\dual{\alg A}=\pair E{ \set{E\setminus X: X\in \alg F}}$ is a convex geometry. 
\end{proposition}

Part of the following proposition
was proved by  \init{P.\,H.~}Edelman~\cite[Theorem 3.3]{edelman}, see also \init{D.~}Armstrong~\cite[Theorem 2.8]{armstrong}. The rest can be extracted from \init{K.~}Adaricheva, \init{V.\,A.~}Gorbunov, and \init{V.\,I.~}Tumanov~\cite[proof of Theorem 1.9]{r:adarichevaetal}. Since this extraction is not so obvious, we will give some details for the reader's convenience. Lattices whose duals are join-distributive were called ``join-semidistributive and lower semimodular'' in \cite{r:adarichevaetal}; here we return to the original terminology of \init{P.\,H.~}Edelman~\cite{edelman}, and call them \emph{meet-distributive}.

\begin{lemma}\label{duallatamatEQ} If $L$ is a meet-distributive lattice and  
 $\alg M=\pair E{\alg G}$ is a  convex geometry, then the following three statements hold.
\begin{enumeratei}
\item\label{latamatEQa}  $\pair{\alg G}{\subseteq}$ is a meet-distributive lattice; it is denoted by  $\halomd(\alg M)$.
\item\label{latamatEQb} $\pair {\Jir L}{\set{\Jir L\cap \ideal x: x\in L }}$ is a  convex geometry; it is denoted by $\geom(L)$.
\item\label{latamatEQc} $\halomd\bigl (\geom(L)\bigr)  \cong L$ and $\geom\bigl(\halomd(\alg M) \bigr)\cong \alg M$.
\end{enumeratei}
\end{lemma}

\begin{proof} \eqref{latamatEQa} was proved by \init{P.\,H.~}Edelman~\cite[Theorem 3.3]{edelman}, see the first paragraph in his proof. 
The proof of \init{K.~}Adaricheva, \init{V.\,A.~}Gorbunov, and \init{V.\,I.~}Tumanov~\cite[Theorem 1.9]{r:adarichevaetal} contains the statement that $\geom(L)$ is a convex geometry. 
The sixteenth line in the proof of \cite[Theorem 1.9]{r:adarichevaetal} explicitely says $\halomd\bigl (\geom(L)\bigr)  \cong L$.

Now, we are left with the proof of $\geom\bigl(\halomd(\alg M) \bigr)\cong \alg M$.
So let $\alg M$ be a  convex geometry. Its closed sets, the members of $\alg G$, are exactly the elements of  $\halomd(\alg M)$. 
We assert that
\begin{equation}\label{siGk}
\text{for every $A\in \alg G$, $\,A$ is 1-generated if{f} $A\in\Jir{(\halomd(\alg M))}$.
}
\end{equation}
To show this, let $A=\set{a_1,\dots,a_t}\in \alg G$. Assume that $A\in\Jir{(\halomd(\alg M))}$. 
The closure operator corresponding to
 $\alg G$ is denoted by  $\Phi$.
Since $A=\Phi(\set{a_1})\vee \dots\vee \Phi(\set{a_t})$ holds in $\halomd(\alg M)$, the join-irreducibility of $A$ yields $A=\Phi(\set{a_i})$ for some $i\in\set{1,\ldots, t}$. This means that $A$ is 1-generated. We prove the converse implication by way of contradiction. Suppose that $A$ is 1-generated but $A\notin\Jir{(\halomd(\alg M))}$. We have 
$A=\Phi(\set{X}) \vee \Phi(\set{Y})$ for some $X,Y\subset A$ such that $A\neq\Phi(\set{X})$ and $A\neq \Phi(\set{Y})$. 
Clearly, we can pick  elements $b_1,\ldots,b_s\in X\cup Y$ such that 
$A=\Phi(\set{b_1})\vee\dots\vee \Phi(\set{b_s})$, this join is  irredundant (that is, no joinand can be omitted), and $s\geq 2$.
Since $A$ is 1-generated, we can also pick a $c\in A$ such that $A=\Phi(\set c)$.
Since the join we consider is irredundant,  $c\notin \Phi(\set{b_1,\dots, b_{s-1}})$,  $b_s\notin \Phi(\set{b_1,\dots, b_{s-1}})$, and $c\neq b_s$. So we have
\[ c\in \Phi\bigl(\Phi(\set{b_1,\dots, b_{s-1}} \cup \set{b_s}\bigr),\quad \set{c,b_s}\cap \Phi(\set{b_1,\dots, b_{s-1}})=\emptyset,\quad\text{and } c\neq b_s\text.
\]
Thus the anti-exchange property yields that $b_s\notin \Phi\bigl(\Phi(\set{b_1,\dots, b_{s-1}} \cup \set{c}\bigr)$. This  contradicts $b_s\in A=\Phi(\set c) \subseteq \Phi\bigl(\Phi(\set{b_1,\dots, b_{s-1}}  \cup \set{c}  \bigr)$, proving  \eqref{siGk}. 

Next,  if we had $x,y\in E$ such that $\Phi(x)=\Phi(y)$ but $x\neq y$, then 
$y\in\Phi(\set y)=\Phi(\set x)= \Phi(\emptyset \cup \set x)$ together with the analogous  $x\in \Phi(\emptyset \cup \set y)$ 
would contradict the anti-exchange property by   Definition~\ref{convgeo}\eqref{convgeoc}. Thus $\Phi(x)=\Phi(y)$  implies  $x=y$. 
Hence, by \eqref{siGk},  for each $A\in\Jir{(\halomd(\alg M))}$, there is a unique $e_A\in E$ such that $A= \Phi(\set{e_A})$. Since $\Phi(\set{e})\in  \Jir{(\halomd(\alg M))}$ also holds for all $e\in E$ by \eqref{siGk}, we have a bijection $\psi\colon  \Jir{(\halomd(\alg M))}\to E$,  defined by $A\mapsto e_A$. Its inverse is denoted by $\eta$; it is defined by $\eta(e)=\Phi(\set e)$.
We assert that $\psi$ is an isomorphism from $\geom\bigl(\halomd(\alg M) \bigr)$ to $\alg M$. To prove this, let $X$ be a closed set of $\geom\bigl(\halomd(\alg M) \bigr)$. This means that $X$ is of the form  $X=\set{A\in \Jir{(\halomd(\alg M))}: A\subseteq B}$ for some $B\in \halomd(\alg M))=\alg G$. Hence, 
\begin{align*}
\psi(X)=\psi\bigl(\set{ \eta(e): \eta(e)\subseteq B }\bigr)= \psi\bigl(\set{ \eta(e): e\in  B } \bigr) = \set{ \psi(\eta(e)): e\in  B } =B\text.
\end{align*}
That is, $\psi$ maps the closed sets $X$ of 
$\geom\bigl(\halomd(\alg M) \bigr)$ to the closed sets of $\alg M$. Since each $B\in \halomd(\alg M))=\alg G$ determines a closed set 
$X=\set{A\in \Jir{(\halomd(\alg M))}: A\subseteq B}$ of $\geom\bigl(\halomd(\alg M) \bigr)$, the calculation above also shows that all $B\in \alg G$ are $\psi$-images of closed sets of $\geom\bigl(\halomd(\alg M) \bigr)$. Thus $\psi$ is an isomorphism.
\end{proof}

Combining Proposition~\ref{pdZT} with the dual of Lemma~\ref{duallatamatEQ}, we easily obtain the following  statement.
It asserts that  join-distributive lattices and antimatroids are essentially the same mathematical objects.

\begin{corollary}\label{latamatEQ} If $L$ is a join-distributive lattice and  
 $\alg A=\pair E{\alg F}$ is an  antimatroid, then the following three statements hold.
\begin{enumeratei}
\item\label{latamatEQa}  $\pair{\alg F}{\subseteq}$ is a join-distributive lattice; it is denoted by  $\halojd(\alg A)$.
\item\label{latamatEQb} $\pair {\Mir L}{\set{\Mir L\setminus \filter x: x\in L }}$ is an antimatroid; it is denoted by $\amat(L)$, 
\item\label{latamatEQc} $\halojd\bigl (\amat(L)\bigr)  \cong L$ and $\amat\bigl(\halojd(\alg A) \bigr)\cong \alg A$.
\end{enumeratei}
\end{corollary}

Now, Corollary~\ref{latamatEQ} allows us to translate Theorem~\ref{thmmain} to a coordinatization of  antimatroids, while the coordinatization of  convex geometries is reduced to that of antimatroids by 
 Proposition~\ref{pdZT}. A brief translation is exemplified by the following corollary; the full translation and the case of convex geometries are omitted.  

Let $\vpi\in S_n^{k-1}$. As before, the set of suborbital vectors and that of $\vpi$-eligible tuples are denoted by $B(\vpi)$ and $\elpi$, respectively; see Definition~\ref{defeligl}, and see also \eqref{suboRbital},  \eqref{dkcBGb}, and \eqref{mlpbP}.  For $\vec x\in\elpi$, let $U(\vec x)$ denote the set $\set{\vec y\in B(\vpi): \vec x\not\leq \vec y}$. The \emph{convex dimension} of an antimatroid $\alg A$ is
the width of $\Jir{(\halojd(\alg A))}$.

\begin{corollary}  For each $\vpi\in S_n^{k-1}$, $\alg A(\vpi)=\pair{B(\vpi)}{\set{U(\vec x) : x\in\elpi }}$ is an  antimatroid with convex dimension at most $k$. Conversely, for each antimatroid $\alg B$ of convex dimension $k$ on an $n$-element set, there exists a  $\vpi\in S_n^{k-1}$ such that $\alg B$ is isomorphic to $\alg A(\vpi)$.
\end{corollary}

\begin{ackno} The author is indebted to Anna Romanowska for calling his attention to 
\cite{abels} and \cite{herscovici}, and to Kira Adaricheva for her comments that led to  \cite{adarichevaczg}.
\end{ackno}

\end{document}